\documentclass[11pt,twoside,reqno]{amsart}
\usepackage{times}
\usepackage{amsmath}
\usepackage{amssymb}
\usepackage{amsfonts}
\usepackage{bbm}
\usepackage{graphicx}
\usepackage{a4wide}
\usepackage{url}

\newcommand{\R}{\mathbbm{R}}
\newcommand{\Z}{\mathbbm{Z}}

\newcommand{\RNonNeg}{\mathbbm{R}_+}
\newcommand{\ints}[1]{[{#1}]}

\DeclareMathOperator{\bigOOp}{O}
\newcommand{\bigO}[1]{\bigOOp({#1})}
\newcommand{\unitVec}[1]{\mathbbm{e}_{#1}}
\newcommand{\oneVec}[1]{\mathbbm{1}_{#1}}
\newcommand{\zeroVec}[1]{\mathbb{O}_{#1}}
\newcommand{\scalProd}[2]{\langle{#1},{#2}\rangle}
\newcommand{\setDef}[2]{\{{#1}:{#2}\}}

\newcommand{\symGr}[1]{\mathfrak{S}({#1})}
\newcommand{\relim}[2]{{#2}({#1})}

\DeclareMathOperator{\kernelOp}{ker}
\newcommand{\kernel}[1]{\kernelOp({#1})}
\DeclareMathOperator{\convOp}{conv}
\newcommand{\conv}[1]{\convOp({#1})}

\DeclareMathOperator{\affOp}{aff}
\newcommand{\aff}[1]{\affOp({#1})}

\newcommand{\reflect}[1]{\varrho^{({#1})}}
\newcommand{\dirtreflect}[1]{\varrho^{\star({#1})}}
\DeclareMathOperator{\hOp}{H}
\newcommand{\hyperPlane}[2]{\hOp^{=}({#1},{#2})}
\newcommand{\halfSpace}[2]{\hOp^{\le}({#1},{#2})}
\DeclareMathOperator{\reflRelOp}{R}
\DeclareMathOperator{\transRelOp}{T}
\DeclareMathOperator{\signRelOp}{S}

\DeclareMathOperator{\domOp}{dom}
\newcommand{\reflRel}[2]{\reflRelOp_{{#1},{#2}}}
\newcommand{\transRel}[2]{\transRelOp_{{#1},{#2}}}
\newcommand{\signRel}[1]{\signRelOp_{{#1}}}
\newcommand{\reflRelHalf}[1]{\reflRelOp_{#1}}

\newcommand{\orth}[1]{{#1}^{\perp}}
\newcommand{\permutahd}[2]{\Pi_{#1}({#2})}
\newcommand{\fundReg}[1]{\Phi_{#1}}
\newcommand{\transp}[2]{\tau_{{#1},{#2}}}
\newcommand{\dirtransp}[2]{\tau^{\star}_{{#1},{#2}}}
\newcommand{\sign}[1]{\sigma_{{#1}}}
\newcommand{\dirtsign}[1]{\sigma^{\star}_{{#1}}}

\DeclareMathOperator{\signtrOp}{sign}
\newcommand{\signtr}[1]{\signtrOp({#1})}
\newcommand{\sort}[1]{{#1}^{(\text{sort})}}
\newcommand{\abs}[1]{{#1}^{(\text{abs})}}
\newcommand{\sortabs}[1]{{#1}^{(\text{sort-abs})}}
\newcommand{\domPnt}[2]{{#1}^{(#2)}}
\DeclareMathOperator{\permOp}{perm}
\DeclareMathOperator{\huffOp}{huff}
\DeclareMathOperator{\polyOp}{P}
\DeclareMathOperator{\VOp}{V}
\DeclareMathOperator{\ctOp}{ct}
\newcommand{\perm}[1]{\polyOp_{\permOp}^{#1}}
\newcommand{\huffPoly}[1]{\polyOp_{\huffOp}^{#1}}
\newcommand{\huffVecs}[1]{\VOp_{\huffOp}^{#1}}

\newcommand{\ctvec}[2]{\ctOp({#1},{#2})}
\newcommand{\ctPoly}[1]{\polyOp_{\ctOp}^{#1}}

\newtheorem{theorem}{Theorem}
\newtheorem{prop}{Proposition}
\newtheorem{rem}{Remark}
\newtheorem{lem}{Lemma}
\newtheorem{obs}{Observation}
\newtheorem{cor}{Corollary}

\begin{document}

\title{Constructing Extended Formulations from Reflection Relations}%
\author{Volker Kaibel}%
\email{kaibel@ovgu.de}%
\author{Kanstantsin Pashkovich}%
\thanks{Supported by the \emph{Max Planck Institute for
Dynamics of Complex Technical Systems}.}
\email{pashkovi@mail.math.uni-magdeburg.de}%

\date{\today}%

 \begin{abstract}
	There are many examples of optimization problems whose associated polyhedra can be described much nicer, and with way less inequalities, by projections of higher dimensional polyhedra than this would be possible in the original space. However, currently not many general tools to construct such extended formulations are available. In this paper, we develop a framework of polyhedral relations that generalizes  inductive constructions of extended formulations via projections, and we particularly elaborate on the special case of reflection relations. The latter ones provide polynomial size extended formulations for several polytopes that can be constructed as convex hulls of the unions of (exponentially) many copies of an input polytope obtained via sequences of reflections at hyperplanes. We demonstrate the use of the framework by deriving small extended formulations for the $G$-permutahedra of all finite reflection groups~ $G$ (generalizing both Goeman's~\cite{Goe09} extended formulation of the permutahedron of size $\bigO{n\log n}$ and Ben-Tal and Nemirovski's~\cite{BN01}  extended formulation with $\bigO{k}$ inequalities for the regular $2^k$-gon) and for  Huffman-polytopes (the convex hulls of the weight-vectors of Huffman codes). 
 \end{abstract}

\maketitle

\section{Introduction}

An \emph{extension} of a polyhedron~$P\subseteq\R^n$ is some polyhedron~$Q\subseteq\R^d$ and a linear projection $\pi:\R^d\rightarrow\R^n$ with $\pi(Q)=P$. A description of~$Q$ by linear inequalities (and equations) is called an \emph{extended formulation} for~$P$. Extended formulations have received quite some interest, as in several cases, one can describe polytopes associated with combinatorial optimization problems much easier by means of extended formulations than by linear descriptions in the original space. In particular, such extensions~$Q$ can have way less facets than the polyhedron~$P$ has. For a nice survey on extended formulations we refer to~\cite{CCZ10}. 

Many fundamental questions on the existence of extended formulations with small numbers of inequalities are open. A particularly prominent one asks whether there are polynomial size extended formulations for the perfect matching polytopes of complete graphs (see~\cite{Yan91,KPT10}). In fact, we lack good techniques to bound the sizes of extended formulations from below, and we also need more tools to construct extended formulations. This paper makes a contribution into the latter direction. 

There are several ways to build extended formulations of polytopes from linear decriptions or from extended formulations of other ones (see, e.g., \cite{MRC90,KL10}). A particular simple way is to construct  them inductively from  extended formulations one has already constructed before. As for an example, let for a vector $p\in\RNonNeg^n$ of \emph{processing times} and for some $\sigma\in\symGr{n}$ (where $\symGr{n}$ is the set of all bijections $\gamma:\ints{n}\rightarrow\ints{n}$ with $\ints{n}=\{1,\dots,n\}$), the \emph{completion time vector} be the vector  $\ctvec{p}{\sigma}\in\R^n$ with
	$\ctvec{p}{\sigma}_j=\sum_{i=1}^{\sigma(j)}p_{\sigma^{-1}(i)}$
for all $j\in\ints{n}$. By some simple arguments (resembling the correctness proof of Smith' rule), one can show that $\ctPoly{p}$ is the image of the polytope $P=\ctPoly{\tilde{p}}\times [0,1]^{n-1}$ for  $\tilde{p}=(p_1,\dots,p_{n-1})\in\R^{n-1}$ under the affine map $f:\R^{2n-2}\rightarrow\R^n$ defined via $f(x)=(x'+p_nx'',\scalProd{\tilde{p}}{\oneVec{}-x''}+p_n)$ with $x=(x',x'')$ and $x',x''\in\R^{n-1}$. 

Applying this inductively, one finds that  $\ctPoly{p}$ is a \emph{zonotope}, i.e.,  an affine projection of a cube of dimension $n(n-1)/2$ (which had already been proved  by Wolsey in the 1980's~\cite{Wol}). 
This may appear surprisingly simple viewing the fact that~$\ctPoly{p}$ has exponentially many facets (see~\cite{Que93}). 
	For the special case 
	 of the \emph{permutahedron}
		$\perm{n}=\ctPoly{\oneVec{n}}=\convOp\setDef{(\gamma(1),\dots,\gamma(n))\in\R^n}{\gamma\in\symGr{n}}$,
	Goemans~\cite{Goe09} found an even smaller extended formulation of size $\bigO{n\log n}$, which we will come back to later.

Let us look again at one  step in the inductive  construction described above. 
With the polyhedron 
\begin{equation}\label{eq:intro:R}
	R=\setDef{(x,y)\in\R^{2n-2}\times\R^{n}}{y=f(x)}\,,
\end{equation}
the extension derived in such a step reads
\begin{equation}\label{eq:intro:PpviaR}
	\ctPoly{p}=\setDef{y\in\R^n}{(x,y)\in R\text{ for some }x\in P}\,.
\end{equation}
Thus, we have derived the extended formulation for~$\ctPoly{p}$ by applying in the sense of~\eqref{eq:intro:PpviaR} the ``polyhedral relation'' defined in~\eqref{eq:intro:R} to a polytope~$P$ of which we had found (inductively) an extended formulation before. The goal of this paper is to generalize this technique of deriving extended formulations by using other ``polyhedral relations'' than graphs of affine maps (which~$R$ as defined in~\eqref{eq:intro:R} is). We will introduce the framework of such general polyhedral relations in Section~\ref{sec:polyRel}, and we are going to elaborate on one particular type of those, called \emph{reflection relations}, in Section~\ref{sec:reflRel}. Reflection relations provide, for  affine halfspaces~$H^{\le}\subseteq\R^n$ and  polyhedra $P\subseteq\R^n$, small extended formulations of the convex hull of the union of~$P\cap H^{\le}$ and the image of $P\cap H^{\le}$ under the orthogonal reflection at the boundary hyperplane of~$H^{\le}$. They turn out to be quite useful building blocks in the construction of some extended formulations. 
We  derive some general results on reflection relations (Theorem~\ref{thm:seqReflRel}) that allow to construct rather easily extended formulations for some particular applications (in particular, without  explicitly dealing with the intermediate polyhedra of iterated constructions) .

In a first application, we show how to derive, for each polytope~$P\subseteq\R^n$ that is contained in (the topological closure of) a region of a finite reflection group~$G$ on~ $\R^n$, an extended formulation of the $G$-permutahedron of~$P$, i.e., the convex hull of the union of the polytopes in the orbit of~$P$ under the action of~$G$ (Section~\ref{subsec:reflGroups}). These extended formulations have  $f+\bigO{n\log n}+\bigO{n\log m}$ inequalities, where~$m$ is the largest number such that $I_2(m)$ appears in the decomposition of~$G$ into  irreducible finite reflection groups, and provided that there is an extended formulation for~$P$ with at most~$f$ inequalities. In particular, this generalizes Goemans' extended formulation of the permutahedron~$\perm{n}$ with $\bigO{n\log n}$ inequalities~\cite{Goe09}. In fact, the starting point of our research was to give an alternative proof for the correctness of Goeman's extended formulation that we would be able to generalize to other constructions. 

As a second application, we provide an extended formulation with $\bigO{n\log n}$ inequalities for the convex hull of all weight-vectors of Huffman-codes with~$n$ words  (Section~\ref{subsec:huffman}). This \emph{Huffman-polytope}~$\huffPoly{n}$ is the convex hull of all vectors $(v_1,\dots,v_n)\in\R^n$ for which there is a rooted binary tree with~$n$ leaves labelled by $1, \dots, n$ such that the distance of leaf~$i$ from the root equals~$v_i$ for all $i\in\ints{n}$. This  provides another striking example of the power of extended formulations, as no linear descriptions of~$\huffPoly{n}$ in~$\R^n$ is known so far, and Nguyen, Nguyen, and Maurras~\cite{NNM10} showed that~$\huffPoly{n}$ has $2^{\Omega(n\log n)}$ facets. 

Two well-known results we obtain easily within the framework of reflection relations are  extended formulations with  $2\lceil \log(m)\rceil+2$ inequalities for regular  $m$-gons   (reproving a result of Ben-Tal and Nemirovski~\cite{BN01}, see Section~\ref{subsubsec:I}) and an extended formulation with $4n-1$ inequalities of the  \emph{parity polytope}, i.e., the convex hull of all $v\in\{0,1\}^n$ with an odd number of one-entries (reproving a result of Carr and Konjevod~\cite{CK04}, see Section~\ref{subsubsec:D}).

We conclude  by  briefly discussing (Section~\ref{sec:concl}) directions for future research on the further extension of the tools presented in this paper .

\paragraph{Acknowledgements} We thank Samuel Fiorini and Michel Goemans for valuable hints and discussions.

%
%
%
%
%

\section{Polyhedral Relations}
\label{sec:polyRel}

A \emph{polyhedral relation} of \emph{type $(n,m)$} is a non-empty polyhedron $\varnothing\ne R\subseteq\R^n\times\R^m$. 
The \emph{image} of a subset~$X\subseteq\R^n$ under such a polyhedral relation~$R$ is denoted by 
\begin{equation*}
	\relim{X}{R}=\setDef{y\in\R^m}{(x,y)\in R \text{ for some }x\in X}\,.
\end{equation*}
Clearly, we have the monotonicity relations $\relim{X}{R}\subseteq \relim{\tilde{X}}{R}$ for $X\subseteq\tilde{X}$. Furthermore,  $\relim{X}{R}$ is a linear projection of  $R\cap(X\times\R^m)$ . Thus, images  of polyhedra and convex sets under polyhedral relations are polyhedra and convex sets, respectively.

A \emph{sequential polyhedral relation} of \emph{type $(k_0,\dots,k_r)$}  is a sequence $(R_1,\dots,R_r)$, where~$R_i$ is a polyhedral relation of type $(k_{i-1},k_i)$ for each $i\in\ints{r}$; its \emph{length} is~$r$.
 For such a sequential polyhedral relation, we denote
by $\mathcal{R}=\mathcal{R}_{(R_1,\dots,R_r)}$ the set of all $(z^{(0)},z^{(r)})\in\R^{k_0}\times\R^{k_r}$ for which there is  some $(z^{(1)},\dots,z^{(r-1)})$ with
	$(z^{(i-1)},z^{(i)})\in R_i$ for all $i\in\ints{r}$.
Note that, since~$\mathcal{R}$ is a linear projection of a polyhedron, $\mathcal{R}$ is a polyhedral relation of type $(k_0,k_r)$. We call $\mathcal{R}_{(R_1,\dots,R_r)}$ the polyhedral relation that is \emph{induced} by the sequential polyhedral relation $(R_1,\dots,R_r)$. 

For a polyhedron $P\subseteq\R^{k_0}$, the polyhedron $Q\subseteq\R^{k_0}\times\cdots\times\R^{k_r}$ defined by
\begin{equation}\label{eq:polyRel:extension}
	z^{(0)}\in P
	\quad\text{and}\quad
	(z^{(i-1)},z^{(i)})\in R_i\quad\text{for all }i\in\ints{r}
\end{equation}
satisfies $\pi(Q)=\relim{P}{\mathcal{R}}$, where~$\pi$ is the projection defined via $\pi(z^{(0)},\dots,z^{(r)})=z^{(r)}$.
Thus, \eqref{eq:polyRel:extension} provides an extended formulation of the polyhedron~$\relim{P}{\mathcal{R}}$ with $k_0+\cdots+k_r$ variables and $f_0+\cdots+f_r$ constraints, provided we have linear descriptions of the polyhedra $P$, $R_1$, \dots, $R_r$ with $f_0$, $f_1$, \dots, $f_r$ constraints, respectively. Of course, one can reduce the number of variables in this extended formulation to~$\dim(Q)$. In order to obtain useful upper bounds on this number by means of the polyhedral relations $R_1$, \dots, $R_r$, let us denote, for any polyhedral relation $R\subseteq\R^n\times\R^m$, by $\delta_1(R)$ and $\delta_2(R)$ the dimension of the non-empty fibers of the orthogonal projection of $\aff{R}$ to the first and second factor of $\R^n\times\R^m$, respectively. If $\aff{R}=\setDef{(x,y)\in\R^n\times\R^m}{Ax+By=c}$, then $\delta_1(R)=\dim(\kernel{B})$ and $\delta_2(R)=\dim(\kernel{A})$. With these parameters, we can estimate
\begin{equation*}
	\dim(Q)\le\min\{k_{0}+\sum_{i=1}^r\delta_1(R_i),k_r+\sum_{i=1}^r\delta_2(R_i)\}\,.
 \end{equation*}

\begin{rem}\label{rem:polyRelEF}
	Let $(R_1,\dots,R_{r})$ be a sequential polyhedral relation of type $(k_0,\dots,k_r)$ with induced polyhedral relation~$\mathcal{R}$, let $\pi:\R^{k_0}\times\cdots\times\R^{k_{r}}\rightarrow\R^{k_r}$ be the projection defined via $\pi(z^{(0)},\dots,z^{(r)})=z^{(r)}$, and let~$f_i$ be the number of facets of~$R_i$ for each $i\in\ints{r}$.
		If the polyhedron~$P\subseteq\R^{k_0}$ has an extended formulation with~$k'$ variables and~$f'$ inequalities, then we can construct an  extended formulation for $\relim{P}{\mathcal{R}}$ with $\min\{k'+\sum_{i=1}^r\delta_1(R_i),k_r+\sum_{i=1}^r\delta_2(R_i)\}$ variables and $f'+f_1+\cdots+f_r$ constraints.
\end{rem}

A particularly simple class of polyhedral relations is defined by polyhedra $R\subseteq\R^n\times\R^m$ with $R=\setDef{(x,y)\in\R^n\times\R^m}{y=f(x)}$ for some affine map $f:\R^n\rightarrow\R^m$.  For these polyhedral relations, a  (linear description of a) polyhedron~$P\subseteq\R^n$ is just an extended formulation of the polyhedron $\relim{P}{R}$ via projection~$f$.

The \emph{domain} of a polyhedral relation $R\subseteq\R^n\times\R^m$ is the polyhedron 
\begin{equation*}
	\domOp(R)=\setDef{x\in\R^n}{(x,y)\in R\text{ for some }y\in\R^m}\,. 
\end{equation*}
We clearly have
	$\relim{X}{R}=\bigcup_{x\in X\cap\domOp(R)}\relim{x}{R}$
for all $X\subseteq\R^n$.
Note that, for a polytope~$P=\conv{V}$ with a finite set $V\subseteq\R^n$ and a polyhedral relation~$R\subseteq\R^n\times\R^m$, in general  the inclusion 
\begin{equation}\label{eq:inclRofVerts}
	\convOp\bigcup_{v\in V}\relim{v}{R}\subseteq\relim{P}{R}
\end{equation}
 holds without equality, even in case of 
$P\subseteq\domOp(R)$; as for an example you may consider $P=\convOp\{0,2\}\subseteq\R^1$ and $R=\convOp\{(0,0),(1,1),(2,0)\}$ with $\relim{P}{R}=[0,1]$ and $\relim{0}{R}=\relim{2}{R}=\{0\}$. Fortunately, one can guarantee equality in~\eqref{eq:inclRofVerts} (which makes it much easier to analyze~$\relim{P}{R}$) for an important subclass of polyhedral relations. 

We call a relation~$R\subseteq\R^n\times\R^m$ 
 \emph{affinely generated} by the  family $(\varrho^{(f)})_{f\in F}$, if~$F$ is finite and every 
 $\varrho^{(f)}:\R^n\rightarrow\R^m$ is an affine map such that
		$\relim{x}{R}=\convOp\bigcup_{f\in F}{\varrho^{(f)}(x)}$
holds for all $x\in\domOp(R)$.
	The maps $\varrho^{(f)}$ ($f\in F$) are called \emph{affine generators} of~$R$ in this case.
	 For such a polyhedral relation~$R$ and a polytope~$P\subseteq\R^n$ with $P\cap\domOp(R)=\conv{V}$ for some $V\subseteq\R^n$, we find
\begin{multline*}
	           \relim{P}{R}
	=          \bigcup_{x\in P\cap\domOp(R)}\relim{x}{R}
	=          \bigcup_{x\in P\cap\domOp(R)}\convOp\bigcup_{f\in F}{\varrho^{(f)}(x)}\\
	\subseteq  \convOp\bigcup_{x\in P\cap\domOp(R)}\bigcup_{f\in F}{\varrho^{(f)}(x)}
	=          \convOp\bigcup_{v\in V}\bigcup_{f\in F}{\varrho^{(f)}(v)}
	\subseteq  \convOp\bigcup_{v\in V}\relim{v}{R}\,,
\end{multline*}
where, due to~\eqref{eq:inclRofVerts},  all inclusions are equations. In particular, we have established the following result.

\begin{prop}\label{prop:polyRel}
	For every polyhedral relation $R\subseteq\R^n\times\R^m$ that is affinely generated by a finite family $(\varrho^{(f)})_{f\in F}$, and for every polytope $P\subseteq\R^n$, we have
	\begin{equation}\label{eq:relimconvunionaffim}
		\relim{P}{R}=\convOp\bigcup_{f\in F}\varrho^{(f)}(P\cap\domOp(R))\,.
	\end{equation}
\end{prop}

As we will often deal with polyhedral relations $\mathcal{R}=\mathcal{R}_{(R_1,\dots,R_r)}$  that are induced by a sequential polyhedral relation $(R_1,\dots,R_r)$, it would be convenient to be able to derive affine generators for~$\mathcal{R}$ from affine generators for $R_1$,\dots,$R_r$. This, however, seems impossible in general, where the difficulties arise from the interplay between images  and domains in a sequence of polyhedral relations. However, one still can derive a very useful analogue of the inclusion ``$\subseteq$'' in~\eqref{eq:relimconvunionaffim}.

\begin{lem}\label{lem:seqAffRel}
 	If $(R_1,\dots,R_{r})$ is a sequential polyhedral relation such that, for each $i\in\ints{r}$, the relation~$R_i$ is affinely generated by the finite family $(\varrho^{(f_i)})_{f_i\in F_i}$, then the inclusion 
\begin{equation*}
	\relim{P}{\mathcal{R}}\subseteq\convOp\bigcup_{f\in F}\varrho^{(f)}(P\cap\domOp(\mathcal{R}))
\end{equation*}
holds for every polyhedron $P\subseteq\R^n$, where $F=F_1\times\cdots\times F_r$ and $\varrho^{(f)}=\varrho^{(f_r)}\circ\cdots\circ\varrho^{(f_1)}$ for each $f=(f_1,\dots,f_r)\in F$.
\end{lem}

We omit the straight-forward proof of Lemma~\ref{lem:seqAffRel} in this extended abstract.

%
%

\section{Reflection Relations}
\label{sec:reflRel}

For $a\in\R^n\setminus\{\zeroVec{}\}$ and $\beta\in\R$, we denote by 
	$\hyperPlane{a}{\beta}=\setDef{x\in\R^n}{\scalProd{a}{x}=\beta}$
the hyperplane defined by the equation $\scalProd{a}{x}=\beta$ and by
	$\halfSpace{a}{\beta}=\setDef{x\in\R^n}{\scalProd{a}{x}\le\beta}$
the halfspace defined by the inequality $\scalProd{a}{x}\le\beta$ (with $\scalProd{v}{w}=\sum_{i=1}^nv_iw_i$ for all $v,w\in\R^n$).
The reflection at $H=\hyperPlane{a}{\beta}$ is 
$\reflect{H}:\R^n\rightarrow\R^n$ where $\reflect{H}(x)$ is the point with $\reflect{H}(x)-x\in\orth{H}$ lying in the one-dimensional linear subspace  $\orth{H}=\setDef{\lambda a}{\lambda\in\R}$ that is orthogonal to~$H$ and $\scalProd{a}{\reflect{H}(x)}=2\beta-\scalProd{a}{x}$. 
The \emph{reflection relation} defined by $(a,\beta)$  is
\begin{equation*}
	\reflRel{a}{\beta}=\setDef{(x,y)\in\R^n\times\R^n}{y-x\in\orth{(\hyperPlane{a}{\beta})}, \scalProd{a}{x}\le\scalProd{a}{y}\le 2\beta-\scalProd{a}{x}}
\end{equation*}
(the definition is invariant against scaling $(a,\beta)$ by positive scalars).
For the halfspace $H^{\le}=\halfSpace{a}{\beta}$, we also denote $\reflRelHalf{H^{\le}}=\reflRel{a}{\beta}$.
 The domain of the reflection relation is
	$\domOp(\reflRel{a}{\beta})=H^{\le}$,
as $(x,y)\in\reflRel{a}{\beta}$ implies $\scalProd{a}{x}\le 2\beta-\scalProd{a}{x}$, thus $\scalProd{a}{x}\le \beta$, and furthermore, for each $x\in\halfSpace{a}{\beta}$, we obviously have $(x,x)\in\reflRel{a}{\beta}$.
Note that, although $(a,\beta)$ and $(-a,-\beta)$ define the same reflection, the reflection relations $\reflRel{a}{\beta}$ and $\reflRel{-a}{-\beta}$ have different domains.

From the constraint $y-x\in\orth{(\hyperPlane{a}{\beta})}$ it follows that $\delta_1(\reflRel{a}{\beta})=1$ holds. 
Thus,  we can deduce the following from Remark~\ref{rem:polyRelEF}.
\begin{rem}\label{rem:sizeSeqRefl}
	If $\mathcal{R}$ is induced by a sequential polyhedral relation of type $(n,\dots,n)$ and length~$r$  consisting of reflection relations only, then, for every polyhedron $P\subseteq\R^n$, an extended formulation of $\relim{P}{\mathcal{R}}$ with $n'+r$ variables and~$f'+2r$ inequalities can be constructed, provided one has at hands an extended formulation for~$P$ with~$n'$ variables and~$f'$ inequalities.
\end{rem}

\begin{prop}\label{prop:polyRel:refl}
	For $a\in\R^n\setminus\{\zeroVec{}\}$, $\beta\in\R$ and the hyperplane $H=\hyperPlane{a}{\beta}$,  the reflection relation $\reflRel{a}{\beta}$ is affinely generated by the identity map and the reflection~$\reflect{H}$. 
\end{prop}

\begin{proof}
	We need to show $\relim{x}{\reflRel{a}{\beta}}=\convOp\{x,\reflect{H}(x)\}$ for every $x\in \domOp(\reflRel{a}{\beta})=\halfSpace{a}{\beta}$. Since, for each such~$x$, we have $(x,x)\in\relim{x}{\reflRel{a}{\beta}}$ and $(x,\reflect{H}(x))\in\relim{x}{\reflRel{a}{\beta}}$, and due to the convexity of $\relim{x}{\reflRel{a}{\beta}}$, it suffices to establish the  inclusion ``$\subseteq$''. 
Thus, let $y\in\relim{x}{\reflRel{a}{\beta}}$ be an arbitrary point in $\relim{x}{\reflRel{a}{\beta}}$. 
Due to $\reflect{H}(x)-x\in\orth{H}$ and $y-x\in\orth{H}$, both $x$ and $\reflect{H}(x)$ are contained in the line $y+\orth{H}$. From $2\beta-\scalProd{a}{x}=\scalProd{a}{\reflect{H}(x)}$ and 
 $\scalProd{a}{x}\le\scalProd{a}{y}\le 2\beta-\scalProd{a}{x}$ we hence conclude that~$y$ is a convex combination of~$x$ and $\reflect{H}(x)$.
\end{proof}

From Proposition~\ref{prop:polyRel} and Proposition~\ref{prop:polyRel:refl}, one obtains the following result.

\begin{cor}\label{cor:polyRel:refl}
	If $P\subseteq\R^n$ is a polytope, then we have, for $a\in\R^n\setminus\{\zeroVec{}\}$ and $\beta\in\R$ defining the hyperplane $H=\hyperPlane{a}{\beta}$ and the halfspace $H^{\le}=\halfSpace{a}{\beta}$,
	\begin{equation*}
		\relim{P}{\reflRel{a}{\beta}}=\convOp\big((P\cap H^{\le})\cup\reflect{H}(P\cap H^{\le})\big)\,.
	\end{equation*}
\end{cor}


While Corollary~\ref{cor:polyRel:refl} describes images under single reflection relations, for  analyses of the images under sequences of reflection relations we define, for each $a\in\R^n\setminus\{\zeroVec{}\}$, $\beta\in\R$, $H^{\le}=\halfSpace{a}{\beta}$, and $H=\hyperPlane{a}{\beta}$, the 
 map $\dirtreflect{H^{\le}}:\R^n\rightarrow\R^n$  via
\begin{equation*}
	\dirtreflect{H^{\le}}(y)=
	\begin{cases}
		y & \text{if } y \in H^{\le}\\
		\reflect{H}(y)              & \text{otherwise}
	\end{cases}
\end{equation*}
for all $y\in\R^n$, which assigns a canonical preimage to every~$y\in\R^n$.
If~$\mathcal{R}$ denotes the polyhedraöl relation induced by the sequential polyhedral relation $(\reflRelHalf{H^{\le}_1},\ldots,\reflRelHalf{H^{\le}_r})$, for all $y\in\R^n$, we have 
\begin{equation}\label{eq:Hstar}
	y\in\relim{\dirtreflect{H^{\le}_1}\circ\cdots\circ\dirtreflect{H^{\le}_r}(y)}{\mathcal{R}}\,.
\end{equation}


\begin{theorem}\label{thm:seqReflRel}
	 Let the  sequential polyhedral relation
	 $(\reflRelHalf{H^{\le}_1},\ldots,\reflRelHalf{H^{\le}_r})$ with halfspaces  $H^{\le}_1, \dots,H^{\le}_r\subseteq\R^n$ and boundary hyperplanes $H_1,\dots,H_r$ induce the polyhedral relation $\mathcal{R}$. 
For  polytopes $P, Q \subseteq\R^n$, with  $Q=\conv{W}$ for some $W\subseteq\R^n$, we have  $Q=\relim{P}{\mathcal{R}}$, whenever the following two conditions are satisfied:
	\begin{enumerate}
		\item \label{cond::seqReflRel1} We have $P\subseteq Q$ and $\reflect{H_i}(Q)\subseteq Q$ for all $i\in\ints{r}$.
		\item \label{cond::seqReflRel2} We have  $\dirtreflect{H^{\le}_1}\circ\cdots\circ\dirtreflect{H^{\le}_r}(w)\in P$ for all $w\in W$.
	\end{enumerate}
\end{theorem}

\begin{proof}
	From the first condition it follows that the image of $P$ under every combination of maps $\reflect{H_i}$ lies in   $Q$. Thus, from Lemma~\ref{lem:seqAffRel} we have the inclusion $\relim{P}{\mathcal{R}}\subseteq Q$. By the second condition  and~\eqref{eq:Hstar}, we have $W\subseteq\relim{P}{\mathcal{R}}$, and hence $Q=\conv{W}\subseteq\relim{P}{\mathcal{R}}$ due to the convexity of $\relim{P}{\mathcal{R}}$.
\end{proof}

In order to provide  simple examples of  extended formulations obtained from reflection relations, let us define the \emph{signing} of a polyhedron $P\subseteq\R^n$ to be
\begin{equation*}
	\signtr{P}=\convOp\bigcup_{\epsilon\in\{-,+\}^n}\epsilon.P\,,
\end{equation*}
where $\epsilon.x$ is the vector obtained from $x\in\R^n$ by changing the signs of all coordinates~$i$ with $\epsilon_i$ being minus. For $x\in\R^n$, we denote by $\abs{x}\in\R^n$ the vector that is obtained from~$x$ by changing every component to its absolute value.

For the construction below we use the reflection relations $\reflRel{-\unitVec{k}}{0}$, denoted by $\signRel{k}$, for all $k\in\ints{n}$. The corresponding reflection  $\sign{k}:\R^n\rightarrow\R^n$ is just the sign change of the $k$-th  coordinate, given by
\begin{equation*}
	\sign{k}(x)_{i}=
	\begin{cases}
		-x_i & \text{if }i=k\\
		x_i & \text{otherwise}
	\end{cases}
\end{equation*}
for all $x\in\R^n$.
The map which defines the canonical preimage with respect to the relation $\signRel{k}$ is given by
\begin{equation*}
	\dirtsign{k}(y)_i=
	\begin{cases}
		|y_i| & \text{if }i=k\\
		y_i & \text{otherwise}
	\end{cases}
\end{equation*}
for all $y\in\R^n$.

\begin{prop}\label{prop:changingSigns}
	If $\mathcal{R}$ is the polyhedral relation that is induced by the sequence $(\signRel{1}, \ldots, \signRel{n})$ and $P\subseteq\R^n$ is a polytope with $\abs{v}\in P$ for each vertex~$v$ of~$P$, then we have
	\begin{equation*}
		\relim{P}{\mathcal{R}}=\signtr{P}\,.
	\end{equation*}
\end{prop}

\begin{proof}
With $Q=\signtr{P}$, the first condition  of Theorem~\ref{thm:seqReflRel} is  satisfied. Furthermore, we 
have $Q=\conv{W}$ with $W=\setDef{\epsilon.v}{\epsilon\in\{-,+\}^n,v\text{ vertex of }P}$. As, for every $w\in W$ with $w=\epsilon.v$ for some vertex $v$ of~$P$ and  $\epsilon\in\{-,+\}^n$, we  have 
$\dirtsign{1}\circ\dots\circ\dirtsign{n}(w)=\abs{w}=\abs{v}\in P$, 
also the second condition of Theorem~\ref{thm:seqReflRel} is satisfied. Hence the claim follows.
\end{proof}

Proposition~\ref{prop:changingSigns} and Remark~\ref{rem:sizeSeqRefl} imply the following.
\begin{theorem}
	For each polytope $P\subseteq\R^n$ with $\abs{v}\in P$ for each vertex~$v$ of~$P$ that admits an extended formulation with~$n'$ variables and~$f'$ inequalities, there is an extended formulation of $\signtr{P}$ with $n'+n$ variables and $f'+2n$ inequalities.
\end{theorem}

\section{Applications}

\subsection{Reflection Groups}
\label{subsec:reflGroups}

A \emph{finite reflection group} is a group~$G$ of finite cardinality that is generated by a (finite) family  $\reflect{H_i}:\R^n\rightarrow\R^n$ ($i\in I$) of reflections at hyperplanes $\zeroVec{}\in H_i\subseteq\R^n$ 
containing the origin. We refer to~\cite{Hum90,FR07} for all results on reflection groups that we will mention. The set of \emph{reflection hyperplanes} $H\subseteq\R^n$ with $\reflect{H}\in G$ (and thus $\zeroVec{}\in H$) --- called the \emph{Coxeter arrangement} of~$G$ --- cuts~$\R^n$ into open connected components, which are called the \emph{regions} of~$G$. The group~$G$ is in bijection with the set of its regions, and it
 acts transitively on these regions. If one distinguishes arbitrarily the topological closure of one of them as the \emph{fundamental domain}  $\fundReg{G}$ of~ $G$, then, for every point $x\in\R^n$, there is a \emph{unique} point
$\domPnt{x}{\fundReg{G}}\in \fundReg{G}$ that belongs to the orbit of $x$ under the action of the group~$G$ on~$\R^n$. 

A finite reflection group~$G$ is called \emph{irreducible} if the set of reflection hyperplanes cannot be partitioned into two sets $\mathcal{H}_1$ and $\mathcal{H}_2$ such that the normal vectors of all hyperplanes in $\mathcal{H}_1$ are orthogonal to the normal vectors of all hyperplanes from $\mathcal{H}_2$.
According to a central classification result, up to linear transformations, the family of irreducible finite reflection groups consists of the four infinite subfamilies $I_2(m)$ (on $\R^2$), $A_{n-1}$, $B_n$, and $D_n$ (on $\R^n$), as well as  six special groups. 

For a finite reflection group~$G$ on~$\R^n$ and some polytope $P\subseteq\R^n$ 
of~$G$, the \emph{$G$-permutahedron} $\permutahd{G}{P}$ of~$P$ is the convex hull of the union of the orbit of~$P$ under the action of~$G$. In this subsection, we show for $G$ being one of $I_2(m)$, $A_{n-1}$, $B_n$, or $D_n$, how to construct an extended formulation for $\permutahd{G}{P}$ from an extended formulation for~$P$. The numbers of inequalities in the constructed extended formulations will be bounded by $f+\bigO{\log m}$ in case of $G=I_2(m)$ and by $f+\bigO{n\log n}$ in the other cases, provided that we have at hands an extended formulation of~$P$ with~$f$ inequalities. By the decomposition  into irreducible finite reflection groups, one  can extend these constructions to  arbitrary finite reflection groups~$G$ on~$\R^n$, where the resulting extended formulations have $f+\bigO{n\log m}+\bigO{n\log n}$ inequalities, where~$m$ is the largest number such that $I_2(m)$ appears in the decomposition of~$G$ into  irreducible finite reflection groups. Details on this will be in the full version of the paper.

\subsubsection{The reflection group $I_2(m)$}
\label{subsubsec:I}
For $\varphi\in\R$, let us denote $H_{\varphi}=\hyperPlane{(-\sin\varphi,\cos\varphi)}{0}$ and $H^{\le}_{\varphi}=\halfSpace{(-\sin\varphi,\cos\varphi)}{0}$. 
The group $I_2(m)$ is generated by the reflections at $H_0$ and $H_{\pi/m}$. It is the symmetry group of the regular $m$-gon with its center at the origin and one of its vertices at $(1,0)$. 
The group $I_2(m)$ consists of the (finite) set of all reflections
$\reflect{H_{k\pi/m}}$ (for $k\in\Z$) and the (finite) set of all rotations around the origin by angles $2k\pi/m$ (for $k\in\Z$).
We choose $\fundReg{I_2(m)}=\setDef{x\in\R^2}{x_2\ge 0,x\in H^{\le}_{\pi/m}}$  as the fundamental domain. 


\begin{prop}\label{prop:I}
	Let~$\mathcal{R}$ be induced by 
	the sequence $(\reflRelHalf{H^{\le}_{\pi/m}},\reflRelHalf{H^{\le}_{2\pi/m}},\reflRelHalf{H^{\le}_{4\pi/m}},\dots,\reflRelHalf{H^{\le}_{2^r\pi/m}})$ of reflection relations with $r=\lceil \log(m)\rceil$. If $P\subseteq\R^2$ is a polytope with $\domPnt{v}{\fundReg{I_2(m)}}\in P$ for 
	each vertex~$v$ of~$P$, then we have
		$\relim{P}{\mathcal{R}}=\permutahd{I_2(m)}{P}$.
\end{prop}

\begin{proof}
	With $Q=\permutahd{I_{2}(m)}{P}$, the first condition  of Theorem~\ref{thm:seqReflRel} is  satisfied. Furthermore, we 
	have $Q=\conv{W}$ with $W=\setDef{\gamma.v}{\gamma\in I_2(m),v\text{ vertex of }P}$. Let $w\in W$ be some point with $w=\gamma.v$ for some vertex $v$ of~$P$ and  $\gamma\in I_2(m)$. Observing that   
	\begin{equation*}
	\dirtreflect{H^{\le}_{\pi/m}}\circ\dirtreflect{H^{\le}_{2\pi/m}}\circ\dots\circ
	\dirtreflect{H^{\le}_{2^r\pi/m}}(w)
	\end{equation*}
	is contained in~$\fundReg{I_2(m)}$, we conclude that it equals $\domPnt{w}{\fundReg{I_2(m)}}=\domPnt{v}{\fundReg{I_2(m)}}\in P$. Therefore, 
	 also the second condition of Theorem~\ref{thm:seqReflRel} is satisfied. Hence the claim follows.
\end{proof}

From Proposition~\ref{prop:I} and Remark~\ref{rem:sizeSeqRefl},  we can conclude the following theorem.
\begin{theorem}\label{thm:I}
	For each polytope $P\subseteq\R^2$ with $\domPnt{v}{\fundReg{I_2(m)}}\in P$ for each vertex~$v$ of~$P$ that admits an extended formulation with~$n'$ variables and~$f'$ inequalities, there is an extended formulation of $\permutahd{I_2(m)}{P}$ with $n'+\lceil \log(m)\rceil +1$ variables and $f'+2\lceil \log(m)\rceil+2$ inequalities.
\end{theorem}

In particular,  we obtain an extended formulation  of a regular $m$-gon with $\lceil \log(m)\rceil+1$ variables and $2\lceil \log(m)\rceil+2$ inequalities by choosing $P=\{(1,0)\}$ in Theorem~\ref{thm:I}, thus reproving 
a result due to Ben-Tal and Nemirovski~\cite{BN01}.

\subsubsection{The reflection group $A_{n-1}$}
\label{subsubsec:An}

The group $A_{n-1}$ is generated by the reflections in~$\R^{n}$ at the hyperplanes~$\hyperPlane{\unitVec{k}-\unitVec{\ell}}{0}$ for all pairwise distinct $k,\ell\in\ints{n}$. It is the symmetry group of the $(n-1)$-dimensional (hence the index in the notation $A_{n-1}$) simplex $\convOp\{\unitVec{1},\dots,\unitVec{n}\}\subseteq\R^n$.
We choose $\fundReg{A_{n-1}}=\setDef{x\in\R^n}{x_1 \le\dots\le x_n}$  as the fundamental domain. The orbit of a point~$x\in\R^n$ under the action of~$A_{n-1}$ consists of all points which can be obtained from~$x$ by permuting coordinates.  Thus the $A_{n-1}$-permutahedron of a polytope $P\subseteq\R^n$ is
\begin{equation*}
	\permutahd{A_{n-1}}{P}=\convOp\bigcup_{\gamma\in\symGr{n}}\gamma.P\,,
\end{equation*}
where $\gamma.x$ is the vector obtained from $x\in\R^n$ by permuting the coordinates according to~$\gamma$.

Let us consider more closely the reflection relation $\transRel{k}{\ell}=\reflRel{\unitVec{k}-\unitVec{\ell}}{0}\subseteq\R^n\times\R^n$.
The corresponding reflection 
$\transp{k}{\ell}=\reflect{H_{k,\ell}}:\R^n\rightarrow\R^n$  with $H_{k,\ell}=\hyperPlane{\unitVec{k}-\unitVec{\ell}}{0}$ 
is  the transposition of coordinates~$k$ and~$\ell$, i.e., we have
\begin{equation*}
	\transp{k}{\ell}(x)_{i}=
	\begin{cases}
		x_{\ell} & \text{if }i=k\\
		x_k & \text{if }i={\ell}\\
		x_i & \text{otherwise}
	\end{cases}
\end{equation*}
for al $x\in\R^n$. The map $\dirtransp{k}{\ell}=\dirtreflect{H_{k,\ell}}:\R^n\rightarrow\R^n$ (assigning canonical preimages) is given by 
\begin{equation*}
	\dirtransp{k}{\ell}(y)=
	\begin{cases}
		\transp{k}{\ell}(y) & \text{if }y_k>y_{\ell}\\
		y                & \text{otherwise}
	\end{cases}
\end{equation*} 
for all $y\in\R^n$.

A sequence $(k_1,\ell_1),\dots,(k_r,\ell_r)\in\ints{n}\times\ints{n}$ with $k_i\ne\ell_i$ for all $i\in\ints{r}$ is called a \emph{sorting network} if 
	$\dirtransp{k_1}{\ell_1}\circ\cdots\circ\dirtransp{k_r}{\ell_r}(y)=\sort{y}$ holds for all $y\in\R^n$, where we denote by $\sort{y}\in\R^n$ the vector that is obtained from~$y$ by sorting the components in non-decreasing order. Note that we have $\domPnt{y}{\fundReg{A_{n-1}}}=\sort{y}$ for all $y\in\R^n$.

\begin{prop}\label{prop:sortingNetworks}
	Let~$\mathcal{R}$ be induced by a sequence $(\transRel{k_1}{\ell_1}, \ldots, \transRel{k_r}{\ell_r})$ of reflection relations, where  $(k_1,\ell_1),\dots,(k_r,\ell_r)\in\ints{n}\times\ints{n}$ is a sorting network. If $P\subseteq\R^n$ is a polytope with $\sort{v}\in P$ for each vertex~$v$ of~$P$, then we have
		$\relim{P}{\mathcal{R}}=\permutahd{A_{n-1}}{P}$.
\end{prop}

\begin{proof}
	With $Q=\permutahd{A_{n-1}}{P}$, the first condition  of Theorem~\ref{thm:seqReflRel} is  satisfied. Furthermore, we 
	have $Q=\conv{W}$ with $W=\setDef{\gamma.v}{\gamma\in\symGr{n},v\text{ vertex of }P}$. As, for every $w\in W$ with $w=\gamma.v$ for some vertex $v$ of~$P$ and  $\gamma\in\symGr{n}$, we  have 
	\begin{equation*}
	\dirtransp{k_1}{\ell_1}\circ\cdots\circ\dirtransp{k_r}{\ell_r}(w)=\sort{w}=\sort{v}\in P\,,
	\end{equation*}
	 also the second condition of Theorem~\ref{thm:seqReflRel} is satisfied. Hence the claim follows.
\end{proof}


As there are sorting networks of size $r=\bigO{n\log n}$ (see~\cite{AKS83}), from Proposition~\ref{prop:sortingNetworks} and Remark~\ref{rem:sizeSeqRefl}  we can conclude the following theorem
\begin{theorem}\label{thm:sortingNetworks}
	For each polytope $P\subseteq\R^n$ with $\sort{v}\in P$ for each vertex~$v$ of~$P$ that admits an extended formulation with~$n'$ variables and~$f'$ inequalities, there is an extended formulation of $\permutahd{A_{n-1}}{P}$ with $n'+\bigO{n\log n}$ variables and $f'+\bigO{n\log n}$ inequalities.
\end{theorem}

Choosing the one-point polytope $P=\{(1,2,\dots,n)\}\subseteq\R^n$, Theorem~\ref{thm:sortingNetworks} yields basically the same extended formulation with $\bigO{n\log n}$ variables and inequalities of the permutahedron
	$\perm{n}=\permutahd{A_{n-1}}{P}$ that has been constructed 
 by Goemans~\cite{Goe09} (see the remarks in the Introduction). 

\subsubsection{The reflection group $B_n$}
The group $B_{n}$ is generated by the reflections in~$\R^{n}$ at the hyperplanes~$\hyperPlane{\unitVec{k}+\unitVec{\ell}}{0}$, $\hyperPlane{\unitVec{k}-\unitVec{\ell}}{0}$ and~$\hyperPlane{\unitVec{k}}{0}$ for all pairwise distinct $k,\ell\in\ints{n}$.
 It is the symmetry group of both the $n$-dimensional cube $\convOp\{-1,+1\}^n$ and the $n$-dimensional cross-polytope $\convOp\{\pm\unitVec{1},\dots,\pm\unitVec{n}\}$.  
We choose $\fundReg{B_{n}}=\setDef{x\in\R^n}{0\le x_1 \le\dots\le x_n}$  as the fundamental domain. The orbit of a point~$x\in\R^n$ under the action of~$B_{n}$ consists of all points which can be obtained from~$x$ by permuting its coordinates and changing the signs of some subset of its coordinates. 	Note that we have $\domPnt{y}{\fundReg{B_{n}}}=\sortabs{y}$ for all $y\in\R^n$, where $\sortabs{y}=\sort{v'}$ with $v'=\abs{v}$.

\begin{prop}\label{prop:Bn}
	Let~$\mathcal{R}$ be induced by a sequence $(\transRel{k_1}{\ell_1}, \ldots, \transRel{k_r}{\ell_r},S_1, \ldots, S_n)$ of reflection relations, where  $(k_1,\ell_1),\dots,(k_r,\ell_r)\in\ints{n}\times\ints{n}$ is a sorting network (and the $S_i$ are defined as at the end of Section~\ref{sec:reflRel}). If $P\subseteq\R^n$ is a polytope with $\sortabs{v}\in P$ for each vertex~$v$ of~$P$, then we have
		$\relim{P}{\mathcal{R}}=\permutahd{B_{n}}{P}$.
\end{prop}

\begin{proof}
	With $Q=\permutahd{B_{n}}{P}$, the first condition  of Theorem~\ref{thm:seqReflRel} is  satisfied. Furthermore, we 
	have $Q=\conv{W}$ with $W=\setDef{\gamma.\epsilon.v}{\gamma\in\symGr{n},\epsilon\in\{-,+\}^n,v\text{ vertex of }P}$. As, for every $w\in W$ with $w=\gamma.\epsilon.v$ for some vertex $v$ of~$P$ and $\gamma\in\symGr{n}$, $\epsilon\in\{-,+\}^n$,  we  have 
	\begin{equation*}
	\dirtransp{k_1}{\ell_1}\circ\cdots\circ\dirtransp{k_r}{\ell_r}\circ\dirtsign{1}\circ\dots\circ\dirtsign{n}(w)=\sortabs{w}=\sortabs{v}\in P\,,
	\end{equation*}
	 also the second condition of Theorem~\ref{thm:seqReflRel} is satisfied. Hence the claim follows.
\end{proof}

As for $A_{n-1}$, we thus can conclude the following from Proposition~\ref{prop:Bn} and Remark~\ref{rem:sizeSeqRefl}.

\begin{theorem}
	For each polytope $P\subseteq\R^n$ with $\sortabs{v}\in P$ for each vertex~$v$ of~$P$ that admits an extended formulation with~$n'$ variables and~$f'$ inequalities, there is an extended formulation of $\permutahd{B_{n}}{P}$ with $n'+\bigO{n\log n}$ variables and $f'+\bigO{n\log n}$ inequalities.
\end{theorem}

\subsubsection{The reflection group $D_n$}
\label{subsubsec:D}
The group $D_{n}$ is generated by the reflections in~$\R^{n}$ at the hyperplanes~$\hyperPlane{\unitVec{k}+\unitVec{\ell}}{0}$ and $\hyperPlane{\unitVec{k}-\unitVec{\ell}}{0}$  for all pairwise distinct $k,\ell\in\ints{n}$. Thus, $D_n$ is a proper subgroup of~$B_n$. It is not the symmetry group of a polytope. 
We choose $\fundReg{D_{n}}=\setDef{x\in\R^n}{|x_1|\le x_2 \le\dots\le x_n}$  as the fundamental domain. The orbit of a point~$x\in\R^n$ under the action of~$D_{n}$ consists of all points which can be obtained from~$x$ by permuting its coordinates and changing the signs of an \emph{even} number of its coordinates. 
For every~$x\in\R^n$, the point $\domPnt{x}{\fundReg{D_n}}$ arises from $\sortabs{x}$ by multiplying the first component by $-1$ in case~$x$ has an odd number of negative components.
For $k,\ell\in\ints{n}$ with $k\ne\ell$, we denote the ordered pair   $(\reflRel{\unitVec{k}-\unitVec{\ell}}{0}, \reflRel{-\unitVec{k}-\unitVec{\ell}}{0})$ of reflection relations by $E_{k,\ell}$. 



\begin{prop}\label{prop:Dn}
	Let~$\mathcal{R}$ be induced by a sequence $(\transRel{k_1}{\ell_1}, \ldots, \transRel{k_r}{\ell_r},E_{1,2},\dots,E_{n-1,n})$ of polyhedral relations, where  $(k_1,\ell_1),\dots,(k_r,\ell_r)\in\ints{n}\times\ints{n}$ is a sorting network. If $P\subseteq\R^n$ is a polytope with $\domPnt{x}{\fundReg{D_n}}\in P$ for each vertex~$v$ of~$P$, then we have
		$\relim{P}{\mathcal{R}}=\permutahd{D_{n}}{P}$.
\end{prop}

\begin{proof}
	With $Q=\permutahd{D_{n}}{P}$, the first condition  of Theorem~\ref{thm:seqReflRel} is  satisfied.
	Let us denote by $\{-,+\}^n_{\text{even}}$ the set of all $\epsilon\in\{-,+\}^n$ with an even number of components equal to minus. Then, we 
	have $Q=\conv{W}$ with $W=\setDef{\gamma.\epsilon.v}{\gamma\in\symGr{n},\epsilon\in\{-,+\}^n_{\text{even}},v\text{ vertex of }P}$. 
	For $k,\ell\in\ints{n}$ with $k\ne\ell$,
	we define $\eta^{\star}_{k,\ell}=\dirtreflect{\halfSpace{\unitVec{k}-\unitVec{\ell}}{0}}\circ\dirtreflect{\halfSpace{-\unitVec{k}-\unitVec{\ell}}{0}}$. For each $y\in\R^n$, the vector $\eta^{\star}_{k,\ell}(y)$ is the vector~$y'\in\{y,\tau_{k,\ell}(y),\rho_{k,\ell}(y),\rho_{k,\ell}(\tau_{k,\ell}(y))\}$ with $|y'_k|\le y'_{\ell}$, where $\rho_{k,\ell}(y)$ arises from~$y$ by multiplying both components~$k$ and~$\ell$ by $-1$.
	As, for every $w\in W$ with $w=\gamma.\epsilon.v$ for some vertex $v$ of~$P$ and $\gamma\in\symGr{n}$, $\epsilon\in\{-,+\}^n_{\text{even}}$,  we  have 
	\begin{equation*}
	\dirtransp{k_1}{\ell_1}\circ\cdots\circ\dirtransp{k_r}{\ell_r}\circ\eta^{\star}_{1,2}\circ\dots\circ\eta^{\star}_{n-1,n}(w)=\domPnt{w}{\fundReg{D_n}}=\domPnt{v}{\fundReg{D_n}}\in P\,,
	\end{equation*}
	 also the second condition of Theorem~\ref{thm:seqReflRel} is satisfied. Hence the claim follows.
\end{proof}

And again, similarly to the cases $A_{n-1}$ and $B_n$, we derive the following result from Proposition~\ref{prop:Dn} and Remark~\ref{rem:sizeSeqRefl}.

\begin{theorem}\label{thm:Dn}
	For each polytope $P\subseteq\R^n$ with $\domPnt{v}{\fundReg{D_n}}(v)\in P$ for each vertex~$v$ of~$P$ that admits an extended formulation with~$n'$ variables and~$f'$ inequalities, there is an extended formulation of $\permutahd{D_{n}}{P}$ with $n'+\bigO{n\log n}$ variables and $f'+\bigO{n\log n}$ inequalities.
\end{theorem}

If we restrict attention to the  polytopes $P=\{(-1,1,\dots,1)\}\subseteq\R^n$ and  $P=\{(1,1,\dots,1)\}\subseteq\R^n$, then we can remove the reflection relations $T_{i_1,j_1},\dots,T_{i_r,j_r}$ from the construction in Proposition~\ref{prop:Dn}. Thus, we obtain extended formulations with $2(n-1)$ variables and $4(n-1)$ inequalities of the convex hulls of all vectors in $\{-1,+1\}^n$ with an odd respectively even number of ones. Thus, applying the affine transformation of  $\R^n$ given by  $y\mapsto \frac12(\oneVec{}-y)$, we derive extended formulations with $2(n-1)$ variables and $4(n-1)$ inequalities for the \emph{parity polytopes} $\convOp\setDef{v\in\{0,1\}^n}{\sum_i v_i\text{ odd}}$ and $\convOp\setDef{v\in\{0,1\}^n}{\sum_i v_i\text{ even}}$, respectively (reproving a result by Carr and Konjevod~\cite{CK04}).

\subsection{Huffman Polytopes}
\label{subsec:huffman}

A vector $v\in\R^n$ (with $n\ge 2$) is a \emph{Huffman-vector} if there is a rooted binary tree with~$n$ leaves (all non-leaf nodes having two children) and a labeling of the leaves by $1,\dots,n$ such that, for each~$i\in\ints{n}$, the number of arcs on the path from the root to the leaf labelled~$i$ equals~$v_i$. Let us denote by $\huffVecs{n}$ the set of all Huffman-vectors in~$\R^n$, and by $\huffPoly{n}=\conv{\huffVecs{n}}$ the \emph{Huffman polytope}. Note that currently no linear description of $\huffPoly{n}$ in $\R^n$ is known. In fact, it seems that such descriptions are extremely complicated. For instance, Nguyen, Nguyen, and Maurras~\cite{NNM10} proved that $\huffPoly{n}$ has $(\Omega(n))!$ facets.

It is easy to see that Huffman-vectors and -polytopes have the following properties.
\begin{obs}\label{obs:huff}\mbox{}
	\begin{enumerate}
		\item For each $\gamma\in\symGr{n}$, we have $\gamma.\huffVecs{n}=\huffVecs{n}$. 
		\item For each $v\in\huffVecs{n}$ there are at least two components of~$v$ equal to $\max\setDef{v_k}{k\in\ints{n}}$.
		\item For each $v\in\huffVecs{n}$  ($n\ge 3$) and $v_i=v_j=\max\setDef{v_k}{k\in\ints{n}}$ for some pair $i<j$, we have  
		\begin{equation*}
			(v_1,\dots,v_{i-1},v_i-1,v_{i+1},\dots,v_{j-1},v_{j+1},\dots,v_n)\in\huffVecs{n-1}\,.
		\end{equation*}
		\item For each $w'\in\huffVecs{n-1}$ ($n\ge 3$), we have  $(w'_1,\dots,w'_{n-2},w'_{n-1}+1,w'_{n-1}+1)\in\huffVecs{n}$.
	\end{enumerate}
\end{obs}

\noindent
For $n\ge 3$, let us define the embedding
\begin{equation*}
	P^{n-1}=\setDef{(x_1,\dots,x_{n-2},x_{n-1}+1,x_{n-1}+1)}{(x_1,\dots,x_{n-1})\in\huffPoly{n-1}}
\end{equation*}
of $\huffPoly{n-1}$ into~$\R^n$.

\begin{prop}
Let $\mathcal{R}\subseteq\R^n\times\R^n$ be the polyhedral relation  that is induced by the following sequence of transposition relations:
\begin{equation}\label{eq:TNSquare}
	\transRel{n-2}{n-1},
	\transRel{n-3}{n-2},
	\dots,
	\transRel{2}{3},
	\transRel{1}{2},
	\transRel{n-1}{n},
	\transRel{n-2}{n-1},
	\dots,
	\transRel{2}{3},
	\transRel{1}{2}
\end{equation}
 Then we have
	 $\mathcal{R}(P^{n-1})=\huffPoly{n}$.
\end{prop}

\begin{proof}
With $P=P^{n-1}$ and $Q=\huffPoly{n}$, the first condition of Theorem~\ref{thm:seqReflRel} is obviously satisfied (due to parts~(1) and~(4) of Observation~\ref{obs:huff}). We have $Q=\conv{W}$ with $W=\huffVecs{n}$. Furthermore, for every $w\in W$ and 
 $x=\tau^{\star}(w)$ with
\begin{equation}\label{eq:tauStarNSquare}
	\tau^{\star}=
	\dirtransp{n-2}{n-1}\circ
	\dirtransp{n-3}{n-2}\circ
	\dots\circ
	\dirtransp{2}{3}\circ
	\dirtransp{1}{2}\circ
	\dirtransp{n-1}{n}\circ
	\dirtransp{n-2}{n-1}\circ
	\dots\circ
	\dirtransp{2}{3}\circ
	\dirtransp{1}{2}\,,
\end{equation}
 we have $x_n=x_{n-1}=\max\setDef{w_i}{i\in\ints{n}}$, hence part~(3) of Observation~\ref{obs:huff} (with $i=n-1$ and $j=n$) implies $\tau^{\star}(w)\in P^{n-1}$. Therefore, the claim follows by Theorem~\ref{thm:seqReflRel}. 
\end{proof}

From Remark~\ref{rem:sizeSeqRefl} we thus obtain an extended formulation for~$\huffPoly{n}$ with $n'+2n-3$ variables and $f'+4n-6$ inequalities, provided we have an extended formulation for~$\huffPoly{n-1}$ with~$n'$ variables and~$f'$ inequalities. As $\huffPoly{2}$ is a single point, we thus can establish inductively the following result.
\begin{cor}
	There are extended formulations of~$\huffPoly{n}$ with $\bigO{n^2}$ variables and inequalities.  
\end{cor}

Actually, one can reduce the size of the extended formulation of $\huffPoly{n}$ to $\bigO{n\log n}$. In order to indicate the necessary modifications, let us denote by $\Theta_k$ the sequence 
\begin{equation*}
	(k-2,k-1),(k-3,k-2),\dots,(2,3),(1,2),(k-1,k),(k-2,k-1),\dots,(2,3),(1,2)
\end{equation*}
of pairs of indices used (with $k=n$) in~\eqref{eq:TNSquare} and~\eqref{eq:tauStarNSquare}. For every sequence $\Theta=((i_1,j_1),\dots,(i_r,j_r))$ of pairs of pairwise different indices, we define 
	$\tau^{\star}_{\Theta}=\dirtransp{i_1}{j_1}\circ\cdots\circ\dirtransp{i_r}{j_r}$
(thus, $\tau^{\star}$ in~\eqref{eq:tauStarNSquare} equals $\tau^{\star}_{\Theta_n}$). Furthermore, we denote by $\eta_k:\R^k\rightarrow\R^{k-1}$ (for $k\ge 3$) the linear map defined via $\eta_k(y)=(y_1,\dots,y_{k-2},y_{k-1}-1)$ for all $y\in\R^k$. The crucial property for the above construction to work is that the following holds for every $v\in \huffVecs{n}$ and $k\in\{3,\dots,n\}$: The vector 
\begin{equation*}
	x=\tau^{\star}_{\Theta_{k}}\circ\eta_{k+1}\circ\tau^{\star}_{\Theta_{k+1}}\circ\cdots\circ
	\eta_n\circ\tau^{\star}_{\Theta_n}(v)
\end{equation*}
satisfies $x_{k-1}=x_{k}=\max\setDef{x_i}{i\in\ints{k}}$. It turns out that this property is preserved when replacing the sequence $\Theta_n$ by an arbitrary sorting network (e.g. of size $\bigO{n\log n}$, see Section~\ref{subsubsec:An}) and, for  $k\in\{3,\dots,n-1\}$, the sequence $\Theta_k$ (of length $2k-3$) by the  sequence
\begin{multline*}
	(i^k_2,i^k_1),(i^k_3,i^k_2),\dots,(i^k_{r(k)-1},i^k_{r(k)-2}),(i^k_{r(k)},i^k_{r(k)-1}),(i^k_{r(k)-1},i^k_{r(k)-2}),\dots,(i^k_3,i^k_2),(i^k_2,i^k_1)
\end{multline*}
with $i^k_1=k$, $i^k_2=k-1$, $i^k_{\ell}=i^k_{\ell-1}-2^{\ell-3}$ for all $\ell\ge 3$, and $r(k)$ being the maximal~$\ell$ with $i^k_{\ell}\ge 1$. As $r(k)$ is bounded by $\bigO{\log k}$ we obtain the following theorem, whose detailed proof will be included in the full version of the paper.

\begin{theorem}
	There are extended formulations of~$\huffPoly{n}$ with $\bigO{n\log n}$ variables and inequalities.  
\end{theorem}

\section{Conclusions}
\label{sec:concl}

We hope to have demonstrated that and how 
the framework of reflection relations extends the currently available toolbox for constructing extended formulations. We conclude with briefly mentioning two directions for future research.

One of the most interesting questions in this context seems to be that for other polyhedral relations that can be useful for constructing extended formulations. In particular, what other types of affinely generated polyhedral relations are there? 

The reflections we referred to  are reflections at hyperplanes. It would be of great interest to find tools to deal with reflections at lower dimensional subspaces as well. This, however, seems to be much harder. In particular, it is unclear whether some concept similar to that of polyhedral relations can help here.

 \bibliographystyle{plain}

\end{document}